\documentclass[9pt]{article}
\usepackage{amsmath}
\usepackage{mathrsfs} 
\usepackage{amsthm}
\usepackage{amsfonts}
\usepackage{amssymb}
\usepackage{latexsym}
\usepackage{tikz} 
\usepackage{esint} 
\usepackage{tikz}
\usepackage{tikz-cd}
\usepackage{cancel}
\usepackage{sectsty}
\usepackage{hyperref}
\usepackage{caption}
\usepackage{float}
\usepackage{setspace}
\usepackage{titlesec}
\titleformat{\subsection}[runin]{\normalfont\bfseries}{\thesubsection}{1em}{}

\usepackage{tocloft} 

\usepackage{euler}

\usepackage[OT2,T1]{fontenc}
\DeclareSymbolFont{cyrletters}{OT2}{wncyr}{m}{n}
\DeclareMathSymbol{\Sha}{\mathalpha}{cyrletters}{"58}

\sectionfont{\fontsize{12}{15}\selectfont}
\subsectionfont{\fontsize{11}{15}\selectfont}
\subsubsectionfont{\fontsize{10}{15}\selectfont}

\usepackage[colorinlistoftodos,textsize=scriptsize]{todonotes}
\usepackage{marginnote}

\newcommand{\mytodo}[2][]{{%
 \let\marginpar\marginnote
 \reversemarginpar
 \renewcommand{\baselinestretch}{0.8}%
 \todo[#1]{#2}}}

\usepackage[matrix,tips,graph,curve]{xy}
\usepackage{enumitem}

\makeatletter

\usepackage[
backend=biber,
style=numeric,
sorting=anyt, doi=false, url=false, isbn=false, giveninits=true
]{biblatex}

\makeatother

\makeatletter 
\@addtoreset{equation}{section}
\makeatother

\theoremstyle{plain}
\newtheorem{theorem}[equation]{Theorem}
\newtheorem{corollary}[equation]{Corollary}
\newtheorem{lemma}[equation]{Lemma}

\theoremstyle{definition}
\newtheorem{definition}[equation]{Definition}

\newtheorem{remark}[equation]{Remark}

\newtheorem{example}[equation]{Example}
\newtheorem{set-up}[equation]{Set-up}

\newcommand{\IC}{\mathbb{C}}
\newcommand{\ID}{\mathbb{D}}

\newcommand{\IF}{\mathbb{F}}
\newcommand{\IG}{\mathbb{G}}

\newcommand{\IQ}{\mathbb{Q}}
\newcommand{\IR}{\mathbb{R}}

\newcommand{\IZ}{\mathbb{Z}}

\newcommand{\End}{\mathrm{End}}
\newcommand{\tr}{\mathrm{tr}}

\newcommand{\Hom}{\mathrm{Hom}}

\newcommand{\Res}{\mathrm{Res}\,}

\newcommand{\ch}{\rm ch} 

\newcommand\iso{\,{\cong}\,} 
\newcommand\tensor{{\otimes}}

\newcommand{\<}{\langle}
\renewcommand{\>}{\rangle}

\newcommand{\into}{\hookrightarrow}

\def\d/{/\mspace{-6.0mu}/}

\def\wt{\widetilde}

\def\what{\widehat}

\newcommand{\w}{\omega}

\newcommand{\Pic}{\mathrm{Pic}\,}
\newcommand{\ord}{\mathrm{ord}}

\newcommand{\Cl}{\mathrm{Cl}}

\newcommand{\Gal}{\mathrm{Gal}}
\newcommand{\NS}{\mathrm{NS}}

\newcommand{\et}{\mathrm{{\acute{e}}t}}

\newcommand{\shA}{\mathscr{A}}

\newcommand{\CSpin}{\mathrm{CSpin}}
\newcommand{\SO}{\mathrm{SO}}

\newcommand{\GSp}{\mathrm{GSp}}
\newcommand{\GL}{\mathrm{GL}}

\newcommand{\Br}{\mathrm{Br}}

\newcommand{\cris}{\mathrm{cris}}

\newcommand{\bpi}{\mathbf{\pi}}
\newcommand{\dR}{\mathrm{dR}}

\newcommand{\Fil}{\mathrm{Fil}}

\newcommand{\ft}{\mathfrak{t}}
\newcommand{\shX}{\mathscr{X}}

\newcommand{\der}{\mathrm{der}}

\newcommand{\shE}{\mathscr{E}}

\newcommand{\Ch}{\mathrm{Ch}}

\newcommand{\LEnd}{\mathrm{LEnd}}

\newcommand{\bL}{\mathbf{L}}

\renewcommand{\O}{\mathrm{O}}

\renewcommand{\hom}{\mathrm{hom}}

\renewcommand{\bpi}{\boldsymbol{\pi}}

\renewcommand{\H}{\mathrm{H}}

\newcommand{\sto}{\stackrel{\sim}{\to}}

\newcommand{\fh}{\mathfrak{h}}

\newcommand{\sfK}{\mathsf{K}}

\newcommand{\shM}{\mathscr{M}}
\newcommand{\cS}{\mathcal{S}}
\newcommand{\PH}{\mathrm{P}}

\newcommand{\fa}{\mathfrak{a}}

\newcommand{\IKC}{\mathsf{IK3}}

\renewcommand{\P}{\mathrm{P}}

\renewcommand{\U}{\mathrm{U}}

\title{{\large{\textbf{The Tate Conjecture for Motivic Endomorphisms of K3 Surfaces over Finite Fields}}}
\vspace{-1ex}
\author{\normalsize{Ziquan Yang}}
\date{\vspace{-5ex}}}
\begin{document}

\maketitle

\begin{abstract}
    The Tate conjecture for squares of K3 surfaces over finite fields was recently proved by Ito--Ito--Koshikawa. We give a more geometric proof when the characteristic is at least $5$. The main idea is to use twisted derived equivalences between K3 surfaces to link the Tate conjecture to finiteness results over finite fields, in the spirit of Tate. 
\end{abstract}

\section{Introduction}

In this note, we give a proof of the Tate conjecture for squares of K3 surfaces over finite fields. We will phrase and study the problem from the perspective of motives, so we first introduce some notation. Let $k$ be a perfect field of characteristic $p$ with a chosen algebraic closure $\bar{k}$, $\shM_{\mathrm{Chow}}(k)$ be the category of Chow motives, which comes equipped with fiber functors $\w_\ell$ (resp. $\w_\cris$) given by $\ell$-adic cohomology for every prime $\ell \neq p$ (resp. crystalline cohomology with $W[1/p]$ coefficients). Let $\shM_{\hom}(k)$ be the category $\shM_{\mathrm{Chow}}(k)$ but with $\Hom$ groups taken modulo the kernel of $\w := \w_\ell \times \w_\cris$. Let $\fh(-)$ denote the usual functor from the category of varieties over $k$ to $\shM_\hom(k)$ given by $X \mapsto (X, \mathrm{id}, 0)$.  

\begin{theorem}
\label{thm: main}
Let $X$ be a K3 surface over a finite field $\IF_q$ of charactersitic $p \ge 5$. Then natural morphism 
$$ \w_\ell : \End(\fh(X)) \tensor \IQ_\ell \to \End_F (\H^*_\et(X_{\bar{\IF}_q}, \IQ_\ell)) $$
is an isomorphism for every prime $\ell \neq p$, where $F$ denotes the Frobenius action on $\H^*_\et(X_{\bar{\IF}_q}, \IQ_\ell)$. 
\end{theorem}
When $X$ is a K3 surface over $k$, $\fh(X)$ is known to admit a Chow-K\"unneth decomposition $\fh(X) = \fh^0(X) \oplus \fh^2(X) \oplus \fh^4(X)$. For two K3 surfaces $X$ and $X'$ over $k$, an \textbf{isogeny} from $X'$ to $X$ is an isomorphism $f : \fh^2(X) \sto \fh^2(X')$ such that $\w(f)$ preserves the Poincar\'e pairing.  We consider the \textbf{isogeny category of K3 surfaces} $\IKC(k)$, whose objects are K3 surfaces over $k$ and whose morphisms are isogenies. For a give K3 surface $X$, we define the \textbf{automorphism group functor of $X$ in the isogeny category} by 
\begin{align*}
    I_X(R) := \{ f\in (\End(\fh^2(X)) \tensor R)^\times : \w(f) \text{ preserves the Poincar\'e pairing} \}
\end{align*}
for every $\IQ$-algebra $R$. Then we have a group-theoretic version of Thm~\ref{thm: main}: 
\begin{theorem} 
\label{thm: group main}
Let $X$ be as in Thm~\ref{thm: main}. Assume that $\IF_q$ is big enough so that all line bundles on $X_{\bar{\IF}_q}$ are defined over $\IF_q$ and $F \in \SO(\H^2_\et(X_{\bar{\IF}_q}, \IQ_\ell(1)))$.  Let $I_{X, \ell}$ denote the centralizer of the geometric Frobenius action in $\SO(\H^2_\et(X_{\bar{\IF}_q}, \IQ_\ell(1)))$. Then $I_X$ is representable by a reductive group of finite type over $\IQ$ and $\w_\ell$ induces an isomorphism $I_X \tensor \IQ_\ell \iso I_{X, \ell}$.
\end{theorem}

\begin{theorem}
\label{thm: existence}
Let $X$ be a K3 surface over $\IF_q$. Let $F$ denote the geometric Frobenius action on $\H^2_\ell(X)$ or $\H^2_\et(X)_\IQ$. For any isometric embedding $\iota : \Lambda_\ell \to \H^2_\ell(X)$, such that $\iota(\Lambda_\ell)$ is $F$-invariant, there exists another K3 surface $X'$ over $\IF_q$, which is unique up to isomorphism and comes with a derived $\ell$-isogeny $f : \fh^2(X') \to \fh^2(X)$ such that $f(\H^2_\ell(X') = \iota(\Lambda_\ell)$. 

Let $X, X'$ be two K3 surfaces over $\IF_q$. Suppose there exists a derived $\ell$-isogeny $f : \fh^2(X') \sto \fh^2(X)$ such that $f(\H^2_\ell(X')) = \H^2_\ell(X)$, then $X \cong X'$. 
\end{theorem}

We now explain the main ideas of proof. First, let $X$ and $\IF_q$ be as in the above theorem and $\ft$ be the transcendental part of $\fh(X)$, i.e., the orthogonal complement to the span of submotives. Similarly, we consider $I_{\ft(X)}$, which is defined just as $I_X$ but with $\fh^2(X)$ replaced by $\ft$ and the centralizer $I_{\ft, \ell}$ of the geometric Frobenius in $\SO(\w_\ell(\ft(1)))$. It suffices to show that $\w_\ell : \End(\ft(1)) \tensor \IQ_\ell \to \End_F (\w_\ell(\ft(1)))$ and $\w_\ell : I_{\ft} \tensor \IQ_\ell \to I_{\ft, \ell}$ are isomorphisms. 

Next, we geometrically construct lots of elements in $\End(\ft)$ using \textit{twisted derived equivalences}. A twisted K3 surface is a pair $(X, \alpha)$ where $X$ is a K3 surface and $\alpha$ is a Brauer class. The order of $(X, \alpha)$ is the order of $\alpha$. Up to choices of B-fields, a derived equivalence $\ID(X, \alpha) \iso \ID(X', \alpha')$ between two twisted K3 surfaces of $\ell^\infty$-order gives rise to an isomorphism $f: \ft \sto \ft'$, where $\ft'$ is the transcendental part of $\fh(X')$. We call compositions of such isomorphisms \textit{derived transcendental $\ell$-isogenies}. Let $S_\ft^{\der, \ell} \subset \End(\ft)$ denote the subgroup of derived transcendental $\ell$-isogenies from $X$ to itself. After developing the structural theorems for derived transcendental $\ell$-isogenies, we use the key fact that there are only finitely many K3 surfaces up to isomorphism over $\IF_q$ to show the quotient $S_\ft^{\der, \ell} \backslash I_{\ft(X), \ell}$ is \textit{compact} (Thm~\ref{thm: fin}). 

Note that this compactness statement already implies that $S_\ft^{\der, \ell}$ is big. Our next step is to find a reductive $\IQ$-group $I'$ which is equipped with a morphism $\nu : I' \to I_{\ft(X), \ell}$ and maps $S_\ft^{\der, \ell} \to I'(\IQ) \to \End(\ft)$ compatible with $\nu$ in the obvious sense. The compactness result and a theorem of Borel and Tis implies that $v(I'_{\IQ_\ell}) \subseteq I_{\ft, \ell}$ is a parabolic subgroup (see Thm~\ref{thm: main lemma}). However, $v(I'_{\IQ_\ell})$ is reductive and $I_{\ft, \ell}$ is connected, so $v(I'_{\IQ_\ell}) = I_{\ft, \ell}$. Then using the density of $I'(\IQ)$ in $I'(\bar{\IQ}_\ell)$ we check that the image of $I'(\IQ)$ in $\End(\ft)$ spans $\End_F \w_\ell(\ft)$. This implies both Thm~\ref{thm: main} and~\ref{thm: group main}. 

It remains to explain how to construct $I'$. Ideally, if $X \times X$ satisfies some of the standard conjectures, then we can directly verify that $I_{\ft}$ defines a reductive $\IQ$-group and simply take $I' = I_{\ft}$. In fact, a little refinement of the argument shows that $\End_F \w_\ell(\ft)$ is spanned by $S^{\der, \ell}_\ft$. We will carry out this strategy out for K3 surfaces which arise from the Kummer construction (Ex.~\ref{exp: Kummer}) and a slight variant of this strategy to ordinary K3 surfaces (Ex.~\ref{exp: ordinary}). To treat the general case, we choose a polarization $\xi$ on $X$ and consider the \textit{Kuga-Satake} abelian variety $A$ attached to $(X, \xi)$. The abelian variety $A$ comes with a ``CSpin-structure''. We consider the group scheme $\wt{I}$ which parametrizes the automorphisms of $A$ which preserve this CSpin-structure and define $I' := \wt{I}/\IG_m$. The main point of considering $\wt{I}$ is that it fixes a polarization on $A$ up to scalar, so that $I'$ is reductive by the \textit{positivity of Rosati involution}. The general mechinery of the Kuga-Satake construction gives us a morphism $I' \to I_{\ft, \ell}$ and we construct maps $S_\ft^{\der, \ell} \to I'(\IQ) \to \End(\ft)$ by lifting-reduction arguments.

\begin{remark}
We remark that Thm~\ref{thm: main} (including the case $p = 2, 3$) was previously proved in \cite{IIK} by Ito--Ito--Koshikawa, whose method relies more heavily on the Kuga-Satake construction. The main idea there is to develop a refined CM lifting theory for K3 surfaces over finite fields, and then deduce Thm~\ref{thm: main} using Kisin's version of Tate's theorem on the endomorphisms of abelian varieties, and the fact that squares of K3 surfaces with CM satisfies the Hodge conjecture, proved by Buskin. In comparison, our method relies more on the geometry of K3 surfaces themselves and Kuga-Satake abelian varieties are only used to bypass the standard conjectures. In particular, we do not make use of any CM liftings or Tate's theorem\footnote{We do need to use the semisimplicity of Frobenius and the divisorial Tate conjecture for individual K3's, whose known proofs indeed depend on Tate's theorem via the Kuga-Satake construction. We will treat these inputs as given.}. There is a long standing connection between finiteness results and the divisorial Tate conjecture, as shown in the work of Tate \cite{Tate} for abelian varieties, and works of Artin--Swinnerton-Dyer \cite{ASD}, Lieblich--Maulik--Snowden \cite{LMS}, and Charles \cite{Charles2} for K3 surfaces. Our main purpose of giving a new proof is to provide an example of such a connection for codimension $2$ cycles, using theory of gerbes and twisted sheaves. 
\end{remark}


\section{Derived $\ell$-Isogenies for K3 Surfaces}

Let $k$ be an algebraically closed field. Let $Y$ be a K3 surface or a product of two K3 surface over $k$ and let $\alpha$ be a Brauer class in $\Br(Y)[\ell^\infty]$. The $\alpha$-twisted coherent sheaves form an abelian category, so we can talk about its bounded derived category $\ID(Y, \alpha)$. To take Chern characters of $\alpha$-twisted sheaves, we need the notion of B-field lifts: 
\begin{definition}
A $B$-field lift of $\alpha$ is an element in $\H^2_\et(Y, \IQ_\ell(1))$ with the following property: Suppose $\ord (\alpha) = \ell^m$ and $B = \beta / \ell^m$. Then $\beta$ lies in $\H^2_\et(Y, \IZ_\ell(1))$ and maps to $\alpha$ under the composition $\H^2_\et(Y, \IZ_\ell(1)) \to \H^2_\et(Y, \mu_{\ell^m}) \to \Br(Y)$, where the first map is the reduction map modulo $\ell^m$. 
\end{definition}
Given a complex of $\alpha$-twisted sheaf $\shE$ on $Y$, and a B-field lift $B$ of $\alpha$ as above, we can define its \textbf{twisted Chern character} $\ch^B(\shE)$, which depends on the choice of $B$. The basic idea is that, given a choice of a gerbe $\shX$ representing $\beta$, $\shE^{\tensor \ell^m}$ becomes an untwisted sheaf on $X$, so we can take the usual Chern character $\ch(\shE^{\tensor \ell^m})$ and define $\ch^B(\shE)$ as $\sqrt[\ell^m]{\ch(\shE^{\tensor \ell^m})}$, which depends only on $B$ and not on $\shX$. For details, see \cite[App.~A]{Bragg-Derived-Equiv} and \cite[\S2]{BY3}. The \textbf{twisted Mukai vector} of $\shE$ is defined by $v^B(\shE) := \ch^B(\shE) \sqrt{\mathrm{td}(Y)}$. 

Now suppose we are given given two twisted K3 surfaces $(X, \alpha), (X', \alpha')$ and a derived equialence $\Phi : \ID(X, \alpha) \sto \ID(X', \alpha')$. By Orlov's theorem, $\Phi$ is induced by a complex $\shE \in D(X \times X', \alpha^{-1} \boxtimes \alpha')$, which is unique up to quasi-isomorphism. To emphasize this dependence we denote $\Phi$ by $\Phi_\shE$. Given B-field lifts $B, B'$ of $\alpha, \alpha'$, we have a twisted Mukai vector $v^{-B \boxtimes B'}(\shE)$. As an element of $\Ch^*(X \times X')_\IQ$, it induces a map $\Phi^{-B \boxtimes B'}(\shE) : \H^*(X)_\IQ \sto \H^*(X')_\IQ$. 

\begin{theorem}
\label{thm: ell prime integrality}
The map $\Phi^{-B \boxtimes B'}(\shE)$ as above restricts to a $\IZ_{\ell'}$-integral isomorphism $\H^*_\et(X, \IZ_{\ell'}) \sto \H^*(X', \IZ_{\ell'})$ for every prime $\ell' \not\in \{\ell, \mathrm{char\,} k\}$. If $\mathrm{char\,} k = p$, then $\Phi^{-B \boxtimes B'}(\shE)$ in addition restricts to a $W$-integral isomorphism $\H^*_\cris(X/W) \to \H^*_\cris(X'/W)$.
\end{theorem}

In general, $\Phi^{-B \boxtimes B'}(\shE) : \H^*(X)_\IQ \sto \H^*(X')_\IQ$ does not need to preserve the codimension filtration, or restricts to an isogeny $\fh^2(X) \sto \fh^2(X')$. Nonetheless, it always restricts to an isomorphism $\ft(X) \sto \ft(X')$ which preserves the Poincar\'e pairing. One argues by Witt's cancellation theorem that $\NS(X)_\IQ$ is isomorphic to $\NS(X')_\IQ$ as quadratic forms, so that given such $\ft(X) \sto \ft(X')$ we can always complete to an isogeny $\fh^2(X) \sto \fh^2(X')$ by adding an isomorphism $\fa^2(X) \sto \fa^2(X')$ which preserves the Poincar\'e pairing. 

\begin{definition}
\begin{enumerate}[label=\upshape{(\alph*)}]
    \item A \textbf{primitive derived transcendental $\ell$-isogeny} is an isomorphism $\ft(X) \sto \ft(X')$ is given by $v^{-B \boxtimes B'}(\shE)$ for some choices of $B, B', \alpha, \alpha'$ and $\Phi_\shE : \ID(X, \alpha) \sto \ID(X', \alpha')$.
    \item A \textbf{primitive derived $\ell$-isogeny} $f : \fh^2(X) \sto 
    \fh^2(X')$ is an isogeny whose restriction to the transcendental parts is a primitive derived transcendental $\ell$-isogeny. 
    \item A (transcendental) \textbf{derived $\ell$-isogeny} is a finite composition of the primitive ones. 
\end{enumerate}
\end{definition}

\begin{remark}
We remark that isomorphisms between K3 surfaces and ``reflection in $(-2)$-curves'' on a single K3 (see \cite[\S6]{Yang}) give rise to primitive derived $\ell$-isogenies with $B, B', \alpha, \alpha'$ all being zero.
\end{remark}

\subsection{Existence Theorems}

\begin{theorem}
\label{thm: existence}
Let $X$ be a K3 surface over $k$. Denote by $\Lambda_\ell$ the quadratic lattice given by $\H^2_\et(X, \IZ_\ell)$. For every isometric embedding $\iota : \Lambda_\ell \into \H^2_\et(X, \IQ_\ell)$, there exists another K3 surface $X'$ together with a derived $\ell$-isogeny $f : \fh^2(X') \to \fh^2(X)$ such that $f(\H^2_\et(X, \IZ_\ell)) = \iota(\Lambda_\ell)$. 
\end{theorem}
\begin{proof}
This is part of \cite[Thm~1.3]{BY3}. We explain the main steps for readers' convenience: First, by Cartan-Dieudonn\'e theorem, every element of $\O(\Lambda_\ell \tensor \IQ_\ell)$ can be written as a composition of relfections, so it suffices to treat the case when $\iota(\Lambda_\ell) = s_b(\H^2_\et(X, \IZ_\ell))$, where $s_b$ is the reflection in $b \in \H^2_\et(X, \IQ_\ell)$. We may assume $b$ is a primitive element in $\H^2_\et(X, \IZ_\ell)$. 

Now set $n = b^2/2$ and $B := b/n$. Let $\alpha$ be the Brauer class defined by $B$. Let $X'$ be the moduli space of stable $\alpha$-twisted sheaves with Mukai vector $(n, 0, 0)$. Then there is an equivalence $\Phi_\shE : \ID(X', \alpha') \sto \ID(X, \alpha)$ for some $\alpha' \in \Br(X')$. For a suitable B-field lift $B'$ of $\alpha'$, $v^{-B' \boxtimes B}(\shE)$ restricts to the desired isogeny $f$. For details of this argument, see \cite[Lem.~6.2(1)]{BY3}.
\end{proof}

When $k = \IC$, this gives an adelic proof of \cite[Thm~0.1]{Huy}, which refines \cite[Thm~1.1]{Buskin}: 
\begin{theorem}
\label{thm: ell Huy}
Let $X, X'$ be algebraic K3 surfaces over $\IC$. Every Hodge isometry $g : \H^2(X', \IQ) \sto \H^2(X, \IQ)$ is given by a derived isogeny.
\end{theorem}
\begin{proof}
By applying Thm~\ref{thm: existence} for different $\ell$'s, we find a derived isogeny $f : \fh^2(X'') \sto \fh^2(X)$ such that $f^*(\H^2(X'', \IZ)) = g(\H^2(X', \IZ))$ for some other K3 surface $X''$. The composition $g \circ f^*$ is an integral Hodge isometry $\H^2(X'', \IZ) \sto \H^2(X', \IZ)$. By \cite[Lem.~6.2]{Yang}, up to precomposing $\pm f$ by a composition of reflections in $(-2)$-curves, we may assume $g \circ f^*$ preserves the ample cones. Therefore, $g \circ f^*$ is induced by an isomorphism by the global Torelli theorem. 
\end{proof}

\subsection{Torelli Theorem}
If a lifting $X_W$ of $X$ over $W$ carries a lifting of all line bundles on $X$, i.e., the natural map $\Pic(X_W) \to \Pic(X)$ is an isomorphism, then we say that $X_W$ is a \textbf{perfect} lifting. By \cite[Prop.~4.10]{LM}, perfect liftings always exist for non-supersingular K3 surfaces. 

\begin{theorem}
\label{thm: lift}
Let $f : \fh^2(X) \sto \fh^2(X')$ be a derived $\ell$-isogeny between non-supersingular K3 surfaces $X, X'$. For every perfect lifting $X_W$ of $X$ over $W$, there exists a perfect lifting $X'_W$ such that $f$ lifts to an isogeny between their generic fibers of $X_W$ and $X_W'$. 
\end{theorem}
\begin{proof}
It suffices to treat the case when $f$ is primitive, so that the restriction of $f$ to the transcendental parts is given by a twisted Fourier-Mukai equivalence $\Phi_\shE : \ID(X, \alpha) \sto \ID(X', \alpha')$ for some Brauer classes $\alpha, \alpha'$ of $\ell$-power order. Moreover, since we are considering perfect liftings, which in particular carry all liftings of $(-2)$-curves, using \cite[Lem.~6.2]{Yang} we reduce to the case when $f$ is polarizable. Any lifting $X_W$ carries a unique lift $\alpha_W$ of the Brauer class $\alpha$, so we may now conclude by \cite[Thm~4.6]{BY3}, which states that for some lifting $X'_W$ of $X$, $\shE$ lifts up to quasi-isomorphism to a complex in $\ID(X_W \times X_W', - \alpha_W \boxtimes \alpha_W')$, where $\alpha'_W$ is the lift of $\alpha'$ on $X_W'$. Moreover, if $X$ is ordinary and $X_W$ is the canonical lifting, then $X'$ is also ordinary and $X'_W$ is the canonical lifting of $X'$. 

For the second statement, note that since $\H^2_\cris(X/W) \iso \H^2_\cris(X'/W)$, $X$ and $X'$ have the same height. Moreover, $\H^2_\et(X_{\bar{K}}, \IZ_p) \iso \H^2_\et(X'_{\bar{K}}, \IZ_p)$ as $\Gal_{K}$-modules. It follows from \cite[Thm~C]{TaelmanOrd} that if $X_W$ is the canonical lifting, then so is $X'_W$. 
\end{proof} 

\begin{theorem}
\label{thm: Torelli}
Let $f : \fh^2(X) \sto \fh^2(X')$  be derived $\ell$-isogeny. If $f$ is polarizable and $\w(f)$ restricts to an isomorphism $\H^2(X) \sto \H^2(X')$, then $f$ is induced by a unique isomorphism $X' \sto X$. 
\end{theorem}

\begin{remark}
Thm~\ref{thm: Torelli} is proved in \cite{BY3} in more generality, which also deals with derived equivlences between twisted K3 surfaces whose Brauer classes are not of prime-to-$p$ order. The proof there also becomes significantly more involved because when the order of $\alpha \in \Br(X)$ is not prime-to-$p$, it is not true that every lifting $X_W$ of $X$ over $W$ carries a lifting of $\alpha$, so the proof of Thm~\ref{thm: lift} breaks down. 
\end{remark}

\subsection{Galois Descent} Using Thm~\ref{thm: Torelli} we give a Galois descent theorem for derived $\ell$-isogenies: 
\begin{theorem}
\label{thm: Gal descent}
Let $k$ be a perfect field with algebraic closure $\bar{k}$. Let $f : \bar{X}' \sto \bar{X}$ be a derived $\ell$-isogeny bewteen two K3 surface $\bar{X}, \bar{X}'$ over $\bar{k}$. Suppose $\bar{X}$ admits a model $X$ over $k$ and the $\IZ_\ell$-lattice $f^*(\H^2_\et(\bar{X}', \IZ_\ell))$ in $\H^2_\et(\bar{X}, \IZ_\ell)$ is invariant under the induced $\Gal_k$-action, then $\bar{X}'$ admits a model $X'$ over $k$ such that $f$ decends to a derived $\ell$-isogeny $\fh^2(X') \sto \fh^2(X)$. 
\end{theorem}
\begin{proof}
For each $\sigma \in \Gal_k$, we use supscript $\sigma$ to denote base change via $\sigma : \bar{k} \to \bar{k}$. Let us write $b_\sigma : \bar \to X$ and $b'_\sigma : (X')^\sigma \to X'$ for the canonical base change morphisms. Suppose $\{ \varphi_\sigma : \bar{X}^\sigma \stackrel{\sim}{\to} \bar{X} \}_{\sigma \in \Gal_k}$ is the descent datum for $\bar{X}$ defined by the model $X$. Our goal is to transport $\{ \varphi_\sigma \}$ to a descent datum for $X'$. Let $\varphi'_\sigma : \fh^2((X')^\sigma) \rightsquigarrow \fh^2(X')$ be the defined by the composition $$ \fh^2((\bar{X}')^\sigma) \stackrel{f^\sigma}{\to} \fh^2(\bar{X}^\sigma) \stackrel{\varphi_\sigma}{\to} \fh^2(\bar{X}) \stackrel{f^{-1}}{\to} \fh^2(\bar{X}').   $$
It is clear that $\varphi'_\sigma$ is a derived $\ell$-isogeny. 

We claim that $\w(\varphi'_\sigma)$ restricts to an isomorphism $\H^2_\et((\bar{X}')^\sigma, \IZ_\ell) \sto \H^2_\et(\bar{X}, \IZ_\ell)$. Since $f$ is an $\ell$-isogeny, it suffices to check that $\varphi'_\sigma$ is $\ell$-integral. Recall that $b_\sigma$ and $b_\sigma'$ induce canonical identifications $\H^2_\et(\bar{X}^\sigma, \IZ_\ell) = \H^2_\et(\bar{X}, \IZ_\ell)$ and $\H^2_\et((X')^\sigma, \IZ_\ell) = \H^2_\et(\bar{X}', \IZ_\ell)$, the action of $\sigma$ on $\H^2_\et(\bar{X}, \IZ_\ell)$ comes exactly from $\varphi_\sigma$. Therefore, the claim follows from the assumption that $f^*(\H^2_\et(\bar{X}', \IZ_\ell))$ is ${\Gal_k}$-stable. 

By Thm~\ref{thm: Torelli}, $\varphi'_\sigma$ is induced by an isomorphism, which we still denote by $\varphi'_\sigma$. Since ${\varphi_\sigma}$ satisfies the cocycle condition, one easily checks using the cohomological rigidity of K3 surfaces (\cite[Prop.~3.4.2]{Rizov1}) that $\{\varphi'_\sigma\}$ also satisfies the cocyle condition. Therefore, $\{ \varphi'_\sigma \}$ defines a descent datum for $X'$, which gives rise to a model $Y'$ over $k$. By construction, the map $f^* : \H^2_\et(\bar{X}', \IQ_\ell) \to \H^2_\et(\bar{X}, \IQ_\ell)$ is ${\Gal_k}$-equivariant for the ${\Gal_k}$-actions induced by $X, X'$. By a standard averaging argument, $f$ descends to an isogeny $\fh^2(X') \sto \fh^2(X)$. 
\end{proof}

\section{Relations between cycle conjectures}
\subsection{Linear Algebra} \label{sec: Linear Algebra} Let $V_\ell$ be a $\IQ_\ell$ vector space of dimension $2r$ equipped with a non-degenerate bilinear pairing $q : V_\ell \times V_\ell \to \IQ_\ell$. Note that $q$ provides an identification of $V_\ell$ with $V_\ell^\vee$, through which induces a natural pairing $q^{\tensor (1, 1)}$ on $\End (V_\ell) \iso V_\ell \tensor V_\ell$ given by the formula
\begin{align}
    \label{eqn: trace formula}
    q^{\tensor (1, 1)}(\alpha, \beta) = \tr(\alpha \beta^\dagger)
\end{align}
where $\beta^\dagger$ denotes the adjoint of $\beta$ under $q$. 

Let $\U(E)$ denote the group scheme defined by units of $E$: For each $\IQ$-algebra $R$, set $\U(E) := (E \tensor_\IQ R)^\times$. In particular, if $E$ is a field extension of $\IQ$, then $\U(E)$ is nothing but the usual Weil restriction $\Res_{E/\IQ} \IG_m$. 

\begin{lemma}
\label{lem: faithfully flat}
If the natural map $E \tensor \IQ_\ell \to \End(V_\ell)$ is injecture, or equivalently, $\dim_\IQ E = \dim_{\IQ_\ell} \IQ_\ell \< E\>$, then the natural map $\U(E) \tensor \IQ_\ell \to \GL(V_\ell)$ is injective.
\end{lemma}
\begin{proof}
When evaluated on a $\IQ_\ell$-algebra $R$, $\U(E) \tensor \IQ_\ell \to \GL(V_\ell)$ becomes the map $(E \tensor_\IQ R)^\times \to \GL(V_\ell \tensor R)$. This map is a restriction of the $R$-linear map $E \tensor_\IQ R \to \End(V_\ell \tensor R)$ which is obtained from $E \tensor \IQ_\ell \to \End(V_\ell)$ via base change. Then we use the fact that algebras over a field are flat. 
\end{proof}

For a quick non-example, note that if $\dim V_\ell = 1$, $E = \IQ(\alpha)$ for some $\alpha \in \IQ_\ell$ which is a root to some irreducible polynomial over $\IQ$ of degree $\ge 2$, then $\U(E) \tensor \IQ_\ell$ cannot inject into $\GL(V_\ell)$ as $\dim \U(E) \ge 2$.  

\begin{lemma}
\label{lem: pairing repellent}
Let $W_\ell$ be any finite dimensional $\IQ_\ell$ vector space equipped with a non-degenerate pairing $W_\ell \times W_\ell \to \IQ_\ell$. Let $W$ be a $\IQ$-vector space embedded in $W_\ell$ such that the natural map $j_\ell : W \tensor \IQ_\ell \to W_\ell$ is surjective. If the pairing on $W_\ell$ restricts to a pairing $W \times W \to \IQ$ on $W$, then $j_\ell$ must also be injective. 
\end{lemma}
\begin{proof}
Choose a section of the projection $W \to W/ (W \cap \ker(j_\ell))$. Then we obtain a finite dimensional subspace $N \subset W$ such that the restriction of $j_\ell$ to $N \tensor \IQ_\ell$ is an isomorphism. In other words, $N$ provides a $\IQ$-linear structure on $W$. Note that the pairing on $N$ which is inherited from $W$ is necessarily non-degenerate. Since $j_\ell$ is surjective, the map $W^\vee \to W_\ell^\vee = (N \tensor \IQ_\ell)^\vee$ is injective. As this map factors as $W^\vee \to N^\vee \to (N \tensor \IQ_\ell)^\vee$, $W^\vee \to N^\vee$ must also be injective, so that $N = W$.  
\end{proof}

\begin{set-up} Let $F \in \SO(V_\ell)$ be a semisimple element such that the set of eigenvalues of $F$ is of the form $\{ \alpha_1^\pm, \alpha_2^\pm, \cdots, \alpha_{r}^\pm \}$ for distinct $\alpha_i$'s in $\bar{\IQ}_\ell$ and none of $\alpha_i$'s is a root of unity. Let $E$ be a vector space over $\IQ$ embedded in $\End (V_\ell)$ such that $E$ is closed under multiplication and $q^{\tensor (1, 1)}$ restricts to a $\IQ$-bilinear pairing $E \times E \to \IQ$ on $E$. Let $I_\ell$ be the centralizer of $F$ in the reductive group $\SO(V_\ell)$ over $\IQ_\ell$. 
\end{set-up}
We first make some elementary observations: Let $W_i^\pm$ be the eigenspace of the eigenvalue $\alpha_i^\pm$. Then $\bar{V}_\ell := V_\ell \tensor \bar{\IQ}_\ell$ admits a decomposition $V_\ell = \oplus_{i = 1}^r (W_i^+ \oplus W_i^-)$. In particular, every $W_i$ (resp. $W_i^-$) is isotropic and only pairs nontrivially with $W_i^-$ (resp. $W_i$). More precisely, $q$ restricts to a perfect pairing $W_i^+ \times W_i^- \to \bar{\IQ}_\ell$ with which we can identify $W_i^+$ with $(W_i^-)^\vee$. There is a natural isomorphism
\begin{align}
\label{eqn: +- embedding}
    \GL(W_i^+) \iso \SO_F(W_i^+ \oplus W_i^-) \subset \End(W_i^+) \times \End(W_i^-)
\end{align}
given by $g \mapsto (g, (g^\vee)^{-1})$. 

\begin{theorem}
\label{thm: main lemma}
Suppose $I$ is a group over $\IQ$ together with a morphism $i: I(\IQ) \to E$ of sets and a morphism $i_\ell:  I \tensor \IQ_\ell \to I_\ell$ of algebraic groups such that the diagram 
\begin{center}
    \begin{tikzcd}
    I(\IQ) \arrow{r}{} \arrow{d}{i} & I(\IQ_\ell) \arrow{d}{i_\ell(\IQ_\ell)} \\
    E \arrow{r}{} \arrow{r}{} & \End V_\ell
    \end{tikzcd}
\end{center}
commutes. If $i(I(\IQ)) \backslash I(\IQ_\ell)$ is compact, then the following are equivalent: {\upshape{(a)}} $I$ is reductive. {\upshape{(b)}} The image of $i_\ell$ is $I_\ell$. {\upshape{(c)}} The image of $I(\IQ)$ in $\End V_\ell$ span $\End_F (V_\ell)$. 
\end{theorem}
\begin{proof}
We first remark that by \cite[Prop.~9.3]{BT}, the compactness of $I(\IQ) \backslash I(\IQ_\ell)$ ensures that $I \tensor \IQ_\ell$ is a parabolic subgroup of $I_\ell$. 

(a) $\Rightarrow$ (b) The quotient $Q := I \tensor \IQ_\ell \backslash I_\ell$ is proper. Since $I \tensor \IQ_\ell$ is reductive, $Q$ must also be affine. The connectedness of $I_\ell$ implies that $Q$ is just a point. 

(b) $\Rightarrow$ (c): Since $I$ is reductive, $I(\IQ)$ is Zariski dense in $I_\ell$. Therefore, it suffices to show that $I(\bar{\IQ}_\ell)$ spans $\End_F(V_\ell)$. This follows easily from the observation that $I_\ell \tensor \bar{\IQ}_\ell = \prod_{i = 1}^r \GL(W_i^+)$ under  (\ref{eqn: +- embedding}). Indeed, it suffices to check that for any vector space $W$ over any field, the image of $\GL(W)$ in $\End(W) \times \End(W)$ under the embedding $g \mapsto (g, g^{-1})$ spans the entire space. 

(c) $\Rightarrow$ (a): Under the isomorphism $I_\ell \tensor \bar{\IQ}_\ell = \prod_{i = 1}^r \GL(W_i^+)$, $I \tensor \bar{\IQ}_\ell$ takes the form $\prod_{i = 1}^r G_i$, where $G_i$ is a parabolic subgroup of $\GL(W_i^+)$. Suppose for some $i$, $G_i$ is a proper subgroup of $\GL(W_i)$. Then $G_i$ has to fix some nontrivial flag in $W_i$, so that $G_i(\bar{\IQ}_\ell) \subset \End(W_i)$ cannot possibly span $\End(W_i)$. This contradicts (c). Therefore, we must have (b), which implies (a). 
\end{proof}

\subsection{Finiteness and Tate Conjecture}
\label{sec: finiteness and Tate}
Let $X$ be a polarized K3 surface over $\IF_q$. Let $F$ be the Frobenius action on $\H^2_\ell(X)(1)$ and assume that $F \in \SO(\H^2_\ell(X)_\IQ(1))$, and the $F$-invariants of $\H^2_\ell(X)_\IQ(1)$ are the same as $F^2$-invariants. Let $\fa$ and $\ft$ denote the algebraic part and the transcendental part of $\fh^2(X)$ respectively. Recall that by the Tate conjecture for K3 surfaces (\cite[Thm~1]{Keerthi}), $\H^2_\ell(X)_\IQ(1)^F$ can be identified with $\NS(X) \tensor \IQ_\ell$. Let $T_\ell(X)$ denote the orthogonal complemenet of $\NS(X)\tensor \IQ_\ell$. Let $I_\ell$ (resp. $I_{\ft, \ell}$) denote the centralizer of $F$ in $\SO(\H^2_\ell(X)_\IQ(1))$ (resp. $\SO(T_\ell(X)_\IQ)$). Note that $I_\ell$ canonically decomposes as $I_{\fa, \ell} \times I_{\ft, \ell}$, where $I_{\fa, \ell} = \SO(\H^2_\ell(X)_\IQ(1)^F)$. Define $I_{\ft, \ell}$ to be the functor which sends every $\IQ$-algebra $R$ to the subgroup of $(\End(\ft) \tensor R)^\times$ which preserves the Poincar\'e pairing $\ft(1) \tensor \ft(1) \to \mathbf{1}$. Let $S^{\der, \ell} \subset \End(\fh^2(X))$ (resp. $S^{\der, \ell}_\ft \subset \End(\ft)$) be the subgroup of derived (resp. derived transcendental) $\ell$-isogenies. 

\begin{theorem}
\label{thm: fin}
The quotient $S^{\der, \ell}_\ft \backslash I_{\ft, \ell}(\IQ_\ell)$ is compact. 
\end{theorem}
\begin{proof}
Note that $I_{\ell} = I_{\fa, \ell} \times I_{\ft, \ell}$. Let $p_\fa(\sfK)$ and $p_\ft(\sfK)$ be the projections of the $I_{\fa, \ell}(\IQ_\ell) \times I_{\ft, \ell}(\IQ_\ell)$ repsectively. Then $\sfK \subseteq p_\fa(\sfK) \times p_\ft(\sfK)$ is a subgroup of finite index. Therefore, it suffices to show that the double quotient $S^{\der, \ell} \backslash I_{\ell}(\IQ_\ell)/ \sfK$ is finite. 

By combining Thm~\ref{thm: existence} and Thm~\ref{thm: Gal descent}, one quickly sees that the elements in the double quotient $S^{\der, \ell} \backslash I_{\ell}(\IQ_\ell)/ \sfK$ corresponds bijectively to the set of isomorphism class of K3 surfaces $X'$ over $\IF_q$ which admits a derived $\ell$-isogeny $f: \fh^2(X') \to \fh^2(X)$. Thanks to \cite[Cor.~2]{Keerthi}, which is a corollary to \cite[Main Theorem]{LMS} and the Tate conjecture for K3's \cite[Thm~1]{Keerthi}, there are only finitely isomorphism classes of K3 surfaces over $\IF_q$. Therefore, $S^{\der, \ell} \backslash I_{\ell}(\IQ_\ell) / \sfK$ is finite. 
\end{proof}

\begin{corollary}
\label{cor: conditional}
Suppose there is a reductive group $I$ over $\IQ$ which is equipped with (set-theoretic) maps $S^{\der, \ell}_\ft \stackrel{j}{\to} I(\IQ) \stackrel{i}{\to} \End(\ft)$ and a morphism $i_\ell: I \to I_{\ft, \ell}$ such that $i_\ell \circ j = \w_\ell$ and $\w_\ell \circ i = i_\ell$. Then $\End_F T_\ell(X)_\IQ$ is spanned by $\w_\ell(i(I(\IQ))$, and $\w_\ell : I_{\ft} \tensor \IQ_\ell \to I_{\ft, \ell}$ is an isomorphism. In particular, Thm~\ref{thm: main} and~\ref{thm: group main} hold for $X$. 
\end{corollary}
\begin{proof}
The first statement is a direct consequence of Thm~\ref{thm: main lemma} and Thm~\ref{thm: fin}. For the second, we first apply Lem.~\ref{lem: pairing repellent} to $W_\ell = \End_F T_\ell(X)_\IQ$ and $W = E$ to obtain the finite-dimensionality of $\End(\ft)$ and the injectivity of $E \tensor \IQ_\ell \to \End_F T_\ell(X)_\IQ$. Then we apply Lem.~\ref{lem: faithfully flat} to $V_\ell = T_\ell(X)_\IQ$ and $E = \End(\ft)$ to conclude that the morphism $I_\ft \tensor \IQ_\ell \to I_{\ft, \ell}$ is injective.  
\end{proof}

\begin{remark}
\label{rmk: conditional}
If we know that (A) the span $\IQ \<S^{\der, \ell}\>$ in $E$ is finite dimensional, and (B) the Poincar\'e pairing on $p_\ft(\IQ\< S^{\der, \ell} \>)$ is negative definite, as predicted by the standard conjectures, then we can simply take $I$ to be the subgroup $I'_{\ft} := \U(\IQ\< S^{\der, \ell}\>) \cap I_{\ft}$ of $I_{\ft}$. Indeed, using (\ref{eqn: trace formula}) one quickly checks that $I'_\ft$ embeds into $\O(p_\ft(\IQ\< S^{\der, \ell}\>))$ via right multiplication. This implies that $I_\ft(\IR)$ is compact, so that $I'_\ft$ is reductive. In fact, in this case, we know that $\End_F T_\ell(X)_\IQ$ is spanned by $S^{\der, \ell}_\ft$. 
\end{remark}

\begin{example}
\label{exp: Kummer} Suppose $X$ comes from the Kummer construction, i.e., $X$ is the desingularization of the quotient $A / A[2]$ for some abelian surface $A$. Set $\wt{A} := X \times_{A/A[2]} A$. Then we have a diagram 
\begin{center}
    \begin{tikzcd}
    & \wt{A}\arrow{dl}{} \arrow{dr}{} & \\
    X & & A
    \end{tikzcd}.
\end{center}
This correspondence induces an isomorphism $\Psi :  \fh^2(A) \oplus (\bL^{\oplus 16}) \iso \fh^2(X)$, where each $\bL$ corresponds to an exceptional divisor, which has self-intersection number $-2$. If $q_X : \fh^2(X)(1)^{\tensor 2} \to \mathbf{1}$ and $q_A: \fh^2(A)(1)^{\tensor 2} \to \mathbf{1}$ denote the Poincar\'e pairings, then $q_X \circ \Psi^{\tensor 2} = 2 q_A$. It is easy to see, using \cite[Thm~1.1]{Ancona} and \cite[Thm~1]{Clozel}, that $X$ satisfies (A) and (B) in Rmk~\ref{rmk: conditional}. 
\end{example}

\begin{example}
\label{exp: ordinary}
Suppose $X$ is an ordinary K3 surface. Then by Thm~\ref{thm: lift} elements in $S^{\der, \ell}$ all lift to the canonical lifting of $X$. Therefore, it follows from Hodge theory that $X$ satisfies (A) and (B) in Rmk~\ref{rmk: conditional}. We remark that the Tate conjecture for any power of an ordinary K3 surface was proved by Zarhin in \cite{ZarhinOrd}. The main point is that the transcendental Tate classes are all spanned by graphs of powers of the Frobenius endomorphism. The $q^{-1}$-multiple of the graph of the $\IF_q$-linear Frobenius endomorphism is indeed an isogeny. However, it cannot be a derived $\ell$-isogeny as long as $X$ is not supersingular. Therefore, for squares of ordinary K3's, our proof yields a different explanation for algebraic cycles spanning the transcendental classes such that the method of constructing these cycles works for arbitrary fields. 
\end{example}

\section{Uncondtional Proof}

\subsection{Kuga-Satake Construction}

Let $\shM_{2d}^\circ$ denote the moduli stack of polarized K3 surface of degree $2d$ over $\IZ_{(p)}$. Let $L_d$ denote the lattice defined in \cite[(3.10)]{Keerthi}. Let $\cS(L_d)$ (resp. $\wt{\cS}(L_d)$) denote the canonical integral model of the associated orthogonal (resp. spinor) Shimura variety. There are sheaves $\bL_B, \bL_{\dR}, \bL_\ell, \bL_p, \bL_{\cris}$ defined on $\cS(L_d)(\IC), \cS(L_d),\cS(L_d), \cS(L_d)_{\IQ}, \cS(L_d)_{\IF_p}$ respectively, with respect to the understood Grothendieck topology (\cite[\S4.1]{Keerthi}). 

\begin{theorem}
There exists an \'etale morphism $\rho : \wt{\shM}_{2d}^\circ \to \cS(L_d)$ such that for every algebraically closed field $\kappa$, $u \in \wt{\shM}_{2d}^\circ(\kappa)$ and $t := \rho(u)$, the following holds: If $(X, \xi)$ is the fiber of the universal family over $u$, then there are canonical isometry $\alpha_* : \bL_{*, t}(-1) \sto \PH^2_*(X)$ for $* = B, \dR, \ell, p, \cris$ whenever the cohomology theory $*$ is applicable. 
\end{theorem}

If $p \mid d$, then $L_{d}$ is not self-dual at $p$. Nonetheless, by \cite[\S6.8]{CSpin} we can find a quadratic lattice $\wt{L}$ which is self-dual at $p$ and of signature $(r, 2)$ for some $r > 19$, together with an isometric embedding $L_d \into \wt{L}$ onto an orthogonal direct summand. This gives rise to maps $\cS(L_d) \to \cS(\wt{L})$. Set $N := (L_d)^\perp \subset \wt{L}$. If $p \nmid d$, then we just set $\wt{L} = L_d$. We additionally consider $\wt{\cS}(\wt{L})$, which is equipped with a polarized abelian scheme $\wt{\shA}$, and sheaves $\wt{\bL}_B, \wt{\bL}_\dR, \wt{\bL}_\ell, \wt{\bL}_{p}, \wt{\bL}_\cris$. To simplify notation, we view each scheme $T$ over $\cS(L_d)$ simultaneously as a scheme over $\cS(\wt{L})$. Then there is a canonical embedding $N \subset \LEnd(\wt{\shA}_T)$ together with compatible isometries $\bL_{*, T} \sto N^\perp \subset \wt{\bL}_{*, T}$ ($*= B, \ell, \dR, p, \cris$ whenever applicable) \cite[Prop.~4.9]{Keerthi}. 

\paragraph{CSpin-structures} Let $\wt{H}$ denote $\Cl(\wt{L})$, viewed as a $\Cl(\wt{L})$-bimodule. Equip $\End(\wt{H}_\IQ)$ with a pairing by $(\alpha, \beta) \mapsto 2^{r + 2} \tr(\alpha \circ \beta)$. Then the map $\wt{L}_\IQ \to \End(\wt{H}_\IQ)$ given by left multiplication is an isometric embedding. Denote by $\pi'$ the idempotent projector from $\End(H'_\IQ)$ onto $\wt{L}_\IQ$. The group $\CSpin(\wt{L}_{\IQ})$ is by definition the stabilizer in $\GL(\wt{H}_\IQ)$ of the $\IZ/2\IZ$-grading, right $\Cl(\wt{L})$-action, and the tensor $\pi'$. We can endow a nondegenerate symplectic pairing $\Psi : \wt{H} \times \wt{H} \to \IZ$ such that the morphism $\CSpin(\wt{L}_\IQ) \to \GL(\wt{H}_\IQ)$ factors through $\GSp(\wt{H}_\IQ, \Psi)$. This gives an embedding of $\wt{\cS}(L_d)$ into a Siegel Shimura variety, from which it inherits a weakly polarized universal abelian scheme $(\wt{\shA}, \lambda')$. Moreover, $\wt{\shA}$ is equipped with a ``CSpin-structure'': a $\IZ/2\IZ$-grading, a left $\Cl(\wt{L})$-action and various realizations of $\pi$. More precisely, for each algebraically closed field$\kappa$, and $\kappa$-point $t$ on $\cS(\wt{L})$ with lift $s$ on $\wt{\cS}(\wt{L})$, there exists an idempotent projector $\wt{\bpi}_{*, s}$ on $\End(\H^1_*(\wt{\shA}_s))$ whose image is canonically identified with $\wt{\bL}_{*, t}$ for $* = B, \dR, \ell, p, \cris$ whenever the cohomology $*$ is applicable. These tensors are compatible in a natural way: If $\mathrm{char\,} \kappa = 0$, the tuple $(\wt{\bpi}_{\dR, s}, \wt{\bpi}_{\what{\IZ}, s})$ is absolute Hodge. In particular, if $\kappa = \IC$, $(\wt{\bpi}_{\dR, s}, \wt{\bpi}_{\what{\IZ}, s})$ comes from the Betti tensor $\pi_{B, s}$. If $\mathrm{char\,} \kappa = p$, $K$ is a finite extension of $W(\kappa)[1/p]$, and $s_K, t_K$ are $K$-points which specialize to $s, t$ respectively, $\wt{\bpi}_{\what{\IZ}^p, s_{\bar{K}}}$ and $\wt{\bpi}_{\what{\IZ}^p, s}$ are compatible via the smooth and proper base change theorem and $\wt{\bpi}_{\cris, s}$ and $\wt{\bpi}_{\dR, s_K}$ are compatible via the Berthelot-Ogus isomorphism $\H^1_\cris(\wt{\shA}_{s_K}) \tensor K \iso \H^1_\dR(\wt{\shA}_{s_K})$. 

Denote the composition $ \wt{\shM}_{2d}^{\circ} \stackrel{\rho}{\to} \cS(L_d) \to \cS(\wt{L})$ by $\rho'$.

\begin{lemma}
\label{lem: filtration}
Let $\kappa$ be a field. Let $t \in \wt{\cS}(\wt{L})(k)$ be a point and $s \in \cS(\wt{L})(k)$ be a lift of the image of $t$. 
\begin{enumerate}[label=\upshape{(\alph*)}]
    \item $\Fil^1 \H^1_\dR(\shA_{s}) = \ker(x)$ for any generator $x$ of the one-dimensional subspace $\Fil^1 \wt{\bL}_{\dR, s}$. Under the inclusion $\bL_\dR \subset \wt{\bL}_\dR$ we have that $\Fil^1 \bL_\dR = \Fil^1 \wt{\bL}_\dR$. 
    \item Let $R$ be any $\IQ_\ell$-algebra. If $\psi_\ell \in \GL(\H^1_\ell(\wt{\shA}_{s}) \tensor R)$ is an element which preserves the $\IZ/2 \IZ$-grading, $\Cl(\wt{L})$-action and $\wt{\bpi}_{\ell}$, then $\psi_\ell$ preserves the symplectic pairing on $\H^1_\ell(\wt{\shA}_{s}) \tensor R$ induced by the weak polarization $\lambda'_s$, which is well defined up to a multiple of $R^\times$. 
\end{enumerate}
\end{lemma}
\begin{proof}
(a) The first statement follows from \cite[Prop.~4.7(v)]{CSpin}. The second statement follows from the fact that $N \subseteq \Fil^0 \wt{\bL}_{\dR}$ is perpendicular to $\Fil^1 \wt{\bL}_\dR$. 
(b) By \cite[Prop.~3.14]{CSpin}, there exists an isomorphism $\eta_\ell : H'_\IQ \sto \H^1_\ell(\shA_s)_\IQ$ which preserves the $\IZ/2 \IZ$-grading, $\Cl(\wt{L})$-action and sends $\pi$ to $\wt{\bpi}_{\ell}$ such that $\eta_\ell$ sends $\Psi$ to the pairing induced by the weak polarization $\lambda'_s$. Then we use the fact that the stablizer of the $\IZ/2 \IZ$-grading, $\Cl(\wt{L})$-action and sends $\pi$ in $\GL(H'_{\IQ_\ell})$ is $\CSpin(\wt{L}_{\IQ_\ell})$, which lies in $\GSp(H'_{\IQ_\ell}, \Psi)$. 
\end{proof}

\subsection{Unconditional Proof}

Let $k$ be a finite field over $\IF_p$. Let $u \in \wt{\shM}^\circ_{2d}(k)$ be a point given by the tuple $(X, \xi, \epsilon)$ for some choice of $\epsilon : \det(L_d \tensor \IZ_2) \sto \det(\PH^2_2(X))$. Set $t = \rho(u)$ and choose a lift $s \in \wt{\cS}(\wt{L})(k)$ of $t$. Up to replacing $k$ by a finite extension, we assume that all line bundles on $X_{\bar{k}}$ and all endomorphisms of $\wt{\shA}_{s, \bar{k}}$ are defined over $k$. Let $S_0^{\der, \ell}$ be the stabilizer of $\xi$ and $\det(\P^2_2(X))$ in $S^{\der, \ell}$. Since $\mathrm{rank\,} \Pic(X) \ge 2$, the restriction of $p_\ft$ to $S_0^{\der, \ell}$ has the same image as $S^{\der, \ell}$. Suppose that $X$ has finite height. 

\begin{lemma}
\label{lem: lift to CSpin-isogeny}
For every $f \in S_0^{\der, \ell}$, there exists a CSpin-isogeny $\psi : \wt{\shA}_s \to \wt{\shA}_s$ which fixes $N_s$ and induces the same action on $ \bL_{*, t}$ for $* = \cris, \what{\IZ}^p$ as $f$ via $\alpha_*$. Moreover, $\psi$ is unique up to a $\IQ^\times$-multiple. 
\end{lemma}
\begin{proof}
By Thm~\ref{thm: lift}, we can find perfect liftings $X_{W}$ and $X'_W$ of $X$ such that $f$ lifts to an isogeny $f_K : \fh^2(X_K) \to \fh^2(X'_K)$ between the generic fibers. Choose an isomorphism $\bar{K} \iso \IC$. Let $t_W, t'_W$ be the liftings of $t$ given by $X_W, X'_W$ and $t_\IC, t_\IC'$ be $t_W \tensor \IC, t'_W \tensor \IC$. Then $f_K \tensor \IC$ gives a rational Hodge isometry $\bL_{B, t_\IC} \tensor \IQ \iso \bL_{B, t'_\IC} \tensor \IQ$. Recall that $N_{t_\IC}$ is canonically identified with $N_{t'_\IC}$, so that can extend $\bL_{B, t_\IC} \tensor \IQ \iso \bL_{B, t'_\IC} \tensor \IQ$ to a Hodge isometry $\wt{\bL}_{B, t_\IC} \tensor \IQ \iso \wt{\bL}_{B, t'_\IC} \tensor \IQ$. Since the map $\wt{\cS}(\wt{L}) \to \cS(L)$ is \'etale, the liftings $t_W, t'_W$ of $t$ can be lifted to liftings $s_W, s'_W$ of $s$ over $W$. Set $s_\IC, s'_\IC$ to be the $\IC$-fibers of $s_W, s'_W$. By \cite[Lem.~5.2]{BY3}, the rational Hodge isometry $\wt{\bL}_{B, t_\IC} \tensor \IQ \iso \wt{\bL}_{B, t'_\IC} \tensor \IQ$ can be lifted to a CSpin-isogeny $\wt{\shA}_{s_\IC} \to \wt{\shA}_{s'_\IC}$. We get the desired $\psi$ by specialzation. 
\end{proof}

\begin{lemma}
\label{lem: reduce to K3 isogeny}
For each CSpin-isogeny $\psi : \wt{\shA}_s \to \wt{\shA}_s$ which preserves $N_s \subset \LEnd(\wt{\shA}_s)$ there exists an isogeny $f : \fh^2(X) \to \fh^2(X')$ which preserves $\xi$ and $\epsilon$ such that $\psi$ and $f$ induce the same actions on $\bL_{*, t} \tensor \IQ$ for $* = \cris, \what{\IZ}^p$. 
\end{lemma}
\begin{proof}
This follows from a slight refinement of the argument as \cite[Prop.~5.2]{Yang}\footnote{There is a typo in \textit{loc. cit.} $p \ge 13$ should be $p \ge 5$.}. For reader's convenience we sketch the argument: First, note that $\psi$ induces an isomorphism $\phi \in \O(\H^2_\cris(X/W)[1/p])$ which preserves the class of $\xi$. By \cite[Lem.~4.5]{Yang}, there exists a finite flat extension $V$ of $W$, and a lifting $X_V$ of $X$ such that $\Pic(X_V) \to \Pic(X)$ is an isomorphism, and the automorphism of $\H^2_\dR(X_V[1/p])$ induced by $\phi$ via the Berthelot-Ogus isomorphism preserves the Hodge filtration. This $X_V$ gives rise to a $V$-valued point $u_V$ which lifts $u$. Set $t_V = \rho'(u_V)$. Since the map $\wt{\cS}(L_d) \to \cS(L_d)$ are \'etale, $u_V$ induces liftings $t_V$ and $s_V$ of $t$ and $s$ over $V$. Next, set $K := V[1/p]$ and recall that $\Fil^1 \H^1_\dR(\shA_{s_K}) = \ker(x)$ for any generator $x$ of the one-dimensional subspace $\Fil^1 \wt{\bL}_{\dR, t_K}$. By Hodge theory, the cycle classes of $N$ in $\wt{\bL}_{\dR, t_K}$ cannot lie in $\Fil^1 \wt{\bL}_{\dR, t_K}$, so we have $\Fil^1 \bL_{\dR, t_K} = \Fil^1 \wt{\bL}_{\dR, t_K}$ under the inclusion $\bL_{\dR, t_K} \subseteq \wt{\bL}_{\dR, t_K}$. By our construction of $s_V$, $\psi$ preserves $\Fil^1 \H^1_\dR(\shA_{s_K})$ via the Berthelot-Ogus isomorphism. This implies that, up to replacing $\psi$ by a $p$-power multiple, $\psi$ lifts to a CSpin-isogeny $\psi_K : \shA_{s_K} \to \shA_{s_K}$, which necessarily preserves $N \subset \LEnd(\shA_{s_K})$. Now choose an isomorphism $\bar{K} \iso \IC$. Then $\psi_K \tensor \IC$ induces an isometry in $\O(\H^2(X_{s_K}(\IC), \IQ))$ which preserves $\xi_K(\IC)$ and the Hodge structure. Finally, apply Thm~\ref{thm: ell Huy} to obtain an auto-isogeny of $X_K \tensor \IC$. Specialize this isogeny to $X$ and we are done. 
\end{proof}

\noindent \textit{Proofs of Thm~\ref{thm: main} and~\ref{thm: group main}.}
Consider group scheme $\wt{I}$ defined by
\begin{align*}
\wt{I}(R) := \{ \psi \in (\End(\wt{\shA}_s) \tensor R)^\times :\, & \psi \text{ respects the $\IZ/2\IZ$-grading $\Cl(\wt{L})$-action and $\pi_* \tensor_\IZ R$ for $* = \cris, \what{\IZ}^p$} \} 
\end{align*}
for each $\IQ$-algebra $R$. Set $I' := \wt{I}/ \IG_m$. It follows from Lem.~\ref{lem: filtration}(b) and the positivity of Rosati involution that $I'(\IR)$ is compact, so that $I'$ is reductive. Note that $I' \tensor \IQ_\ell$ admits a natural morphism into the centralizer of $F$ in $\SO(\bL_{\bar{s}, \ell}) \stackrel{\alpha_\ell}{=} \SO(\P^2_\et(X, \IQ_\ell(1)))$. Let $\ft$ denote $\ft(X)$ and let $I_\ft$ be as defined in \S~\ref{sec: finiteness and Tate}. Since $I'(\IQ) = \wt{I}(\IQ)/ \IG_m(\IQ)$ by Hilbert theorem 90, Lem.~\ref{lem: lift to CSpin-isogeny} and~\ref{lem: reduce to K3 isogeny} give us maps $S_0^{\der, \ell} \to I'(\IQ) \to I_{\ft}(\IQ)$ which are compatible with the morphisms to $I_{\ft}(\IQ_\ell)$. It is easy to check that $p_\ft(S_0^{\der, \ell}) = S_\ft^{\der, \ell}$ is surjective, so there exists at least a set-theoretic section $S_\ft^{\der, \ell} \to S_0^{\der, \ell}$. Then we conclude by Cor.~\ref{cor: conditional}. \qed

\printbibliography

\end{document}